\documentclass{amsart}[12pt]
\usepackage[colorlinks=true,          % link colors, set to 'false' for print version
            linkcolor=blue,
            citecolor=red,
            urlcolor=blue]{hyperref}
\usepackage{amssymb,amsfonts,amsmath,hyperref,mathrsfs,latexsym,stmaryrd,graphicx}
\DeclareMathAlphabet{\mathpzc}{OT1}{pzc}{m}{it}

\input xy
\xyoption{all}

\newcommand {\Def}{\textrm{Def}}
\newcommand {\MC} {\textrm{MC}}
\newcommand {\Art}{\textrm{Art}_\CC}

\newcommand {\CC}{\mathbb{C}}

\newcommand{\NN}{\mathbb{N}}

\newcommand {\bp}{{\bf p}}

\newcommand {\bG}{{\bf G}}

\newcommand {\bP}{{\bf P}}

\newcommand{\cA}{\mathcal{A}}
\newcommand{\cB}{\mathcal{B}}

\newcommand{\cL}{\mathcal{L}}

\newcommand {\cO}{\mathcal{O}}

\newcommand{\cK}{\mathcal{K}}

\newcommand{\scA}{\mathscr{A}}
\newcommand{\scB}{\mathscr{B}}
\newcommand{\scC}{\mathscr{C}}

\newcommand{\scP}{\mathscr{P}}

\newcommand{\scZ}{\scZ}

\newcommand{\fg}{\mathfrak{g}}

\newcommand{\fl}{\mathfrak{l}}
\newcommand{\fm}{\mathfrak{m}}

\newcommand {\dbar}{\overline{\partial}}

\newcommand {\mhom}{\textrm{Hom}}

\newcommand{\sets}{\textrm{Sets}}
\newcommand {\ad}{\textrm{ad} }

\newcommand  {\eps}{\varepsilon}

\newtheorem{thm}{Theorem}[section]
\newtheorem{Proposition}[thm]{Proposition}
\newtheorem{Lemma}[thm]{Lemma}
\newtheorem{corollary}[thm]{Corollary}

\newtheorem*{thma}{{\bf Theorem A}}
\newtheorem*{thmb}{{\bf Theorem B}}

\begin{document}
\title{ On the L-infinity description of the Hitchin Map}
\author{Peter Dalakov}
\email{pdalakov@sissa.it}
\address{SISSA, via Bonomea 265, 34136 Trieste, Italy. On leave from IMI-BAS, Sofia, Bulgaria.}


\begin{abstract}
We exhibit, for a $G$-Higgs bundle on a compact complex manifold, a subspace of the second cohomology of the
controlling dg Lie algebra,  containing the  obstructions to smoothness.  For this we  construct an $L_\infty$-morphism,
which induces the Hitchin map and whose  ``toy  version''  controls the adjoint quotient morphism.
This  extends  recent results of E. Martinengo.
\end{abstract}
\maketitle
\section{ Introduction}
The idea that
deformation problems  are controlled by differential graded Lie algebras (dgla's) has been
a key  guiding principle in (characteristic zero) deformation theory.
This  philosophy, currently subsumed in the work of J.Lurie, has been actively exploited by
  Deligne, Drinfeld, Gerstenhaber, Goldman--Millson, 
Kontsevich,
Nijenhuis--Richardson, Schlessinger, Stasheff and     Quillen, since the earliest days of the subject.
 Kontsevich argued in \cite{kontsevich_def_quant} that
the formal geometry
of moduli problems is governed by  a richer structure: an $L_\infty$-algebra (strongly homotopy Lie algebra), 
and natural transformations between deformation functors are induced by $L_\infty$-morphisms
of  the controlling dgla's.
For example, it is shown in \cite{manetti_fiorenza_cartan} that Griffiths' period map is induced by an
$L_\infty$-morphism.
 In a certain sense, an $L_\infty$-morphism encodes the ``Taylor expansion''
of a morphism of pointed formal varieties (\cite{kontsevich_def_quant}, \S 4.1). By general
deformation-theoretic arguments its linear part is a morphism of obstruction theories.

In view of this, given a pair of deformation functors and a natural transformation  between them, 
 one is  confronted with the questions of 
identifying controlling dgla's and a corresponding $L_\infty$-morphism. Apart from being aesthetically pleasing, this  gives 
additional information for the
obstruction spaces of the two functors. The main result in  this note is a
particular example of such a  setup.

Let $X$ be a compact complex manifold, and
 $G$ be a complex reductive Lie group of rank $N$, with Lie algebra $\fg$. By an \emph{$\Omega^1_X$-valued $G$-Higgs bundle  (Higgs pair) on $X$} we shall mean  
 a  pair $(\bP,\theta)$, consisting of a
  holomorphic principal $G$-bundle  $\bP\to X$, and a section $\theta\in H^0(X,\ad \bP\otimes \Omega_X^1)$, satisfying $\theta\wedge\theta=0$. 
The Hitchin map  associates to $(\bP,\theta)$  the spectral invariants of $\theta$. After  some choices, these invariants determine
 a point in $\cB=\bigoplus_{d_i} H^0(X,S^{d_i}\Omega^1_X)$, where $d_i$ are the degrees of the basic $G$-invariant  polynomials
on $\fg$.
 For example, if $G$ is a classical group, then $\theta$ can be  represented  locally on $X$ by a matrix of holomorphic 1-forms, with commuting components.
Then the Hitchin map  assigns to it
 the coefficients of its  characteristic polynomial.
Considering  Higgs pairs on $X\times \textrm{Spec}A$, for  an Artin local ring $A$, allows one to define
the Hitchin map as a natural transformation between suitable deformation functors, see Section \ref{dgla}.

If $X$ is projective (\cite{moduli2}) or compact K\"ahler (\cite{fujiki})
there exist actual (coarse) moduli spaces of (semi-stable) Higgs pairs with fixed topological invariants.
 In the projective case  the Hitchin map
is known to be a proper morphism to $\cB$ (\cite{moduli2}, \S 6), and if $\dim_\CC X=1$ it  determines  an algebraic completely integrable
Hamiltonian system (\cite{hitchin_sd}, \cite{hitchin_sb}).

The main resluts in this note are the following two theorems.

\begin{thma}\label{obstruct}
Let $X$ be a compact complex  manifold, $G$ a complex reductive Lie  group, and $\{p_i, i\in E\}$  homogeneous generators of
$\CC[\fg]^G$, $\deg p_i=d_i$. Let $(\bP,\theta)$ be a $G$-Higgs bundle on $X$,
and $\scC^\bullet= \bigoplus _{p+q=\bullet} A^{0,p}(\ad \bP\otimes \Omega^q_X)$ its controlling dgla. Then the obstruction space
$O_{\Def_{\scC^\bullet}}\subset H^2(\scC^\bullet)$ is contained in the kernel of the map
\[
  H^2(\scC^\bullet)\longrightarrow   \bigoplus_{i\in E} H^1(X,S^{d_i}\Omega^1_X)  
\]
\[
\left[ s^{2,0},s^{1,1},s^{0,2}\right] \longmapsto \bigoplus_{i\in E}  (\partial p_i)(s^{1,1}\otimes \theta^{d_i-1}).  
\]
Here $\partial p_i$ denotes the differential of $p_i$, thought of as an element of $\fg^\vee\otimes S^{d_i-1}(\fg^\vee)$.
\end{thma}
We make some remarks about the geometrical meaning of this theorem in Section \ref{obstructions}. {\bf Theorem A} is an easy consequence of
the more technical 
\begin{thmb}\label{hitchin2}
Let $X$ be a compact complex manifold, $G$ a complex reductive Lie group, and $\{p_i, i\in E\}$  homogeneous generators of
$\CC[\fg]^G$, $\deg p_i=d_i$. Let $(\bP,\theta)$ be a $G$-Higgs bundle on $X$,
and $\scC^\bullet= \bigoplus _{p+q=\bullet} A^{0,p}(\ad \bP\otimes \Omega^q_X)$ its controlling dgla.
 Let
 $\bp_0: \overline{S^\bullet(\scC^\bullet)}\to \overline{S^\bullet\left(A^{0,\bullet}(\ad\bP\otimes \Omega^1_X)\right)}$
be the homomorphism
induced by $\bigoplus_{p,q}s^{p,q}\mapsto s^{1,q}\in  A^{0,q}(\ad\bP\otimes \Omega^1_X)$.
 Then the collection of maps
\[
\xymatrix@1{\bigoplus_{i} h^{d_i}_k=\bigoplus_{i}(\partial^k p_i)(_{-}\otimes \theta^{d_i-k})\circ \bp_0:& S^k(\scC^\bullet[1])\ar[r]\ &\ \bigoplus_{i} A^{0,\bullet}(S^{d_i}\Omega^1_X)      }
\]
\[\bigoplus_{p_1,q_1} s_1^{p_1,q_1}\cdot \bigoplus_{p_2,q_2} s_2^{p_2,q_2}\cdot\ldots \cdot\bigoplus_{p_k,q_k} s_k^{p_k,q_k} 
\longmapsto \bigoplus_{i} \sum_{q_1,\ldots,q_k} (\partial^k p_i)( s_1^{1,q_1}\otimes \ldots\otimes  s_k^{1,q_k}\otimes \theta^{d_i-k}) 
\]
for all $k\geq 1$,   induces an $L_\infty$-morphism
\[
h_\infty: \scC^\bullet = \bigoplus _{p+q=\bullet} A^{0,p}(\ad \bP\otimes \Omega_X^q)\to \scB^\bullet =\bigoplus _{i\in E} A^{0,\bullet}(S^{d_i}\Omega^1_X)[-1].
\]
The   natural transformation of deformation functors,  induced by $h_\infty$ is the Hitchin map: $\Def(h_\infty)=H$
under the identifications $\Def_{\scB^\bullet}\simeq \Def_{H(E,\theta)}$ and $\Def_{\scC^\bullet}\simeq \Def_{(\bP,\theta)}$.
\end{thmb}
The content of this note is organised as follows.
In Section 2 we discuss dgla's and $L_\infty$-algebras, and give some examples.
In Section 3  we  study a Lie-algebraic  ``toy model'' for the Hitchin map.
For that,  we fix   homogeneous generators of $\CC[\fg]^G$, which allows us to identify
the adjoint quotient morphism $\fg\to \fg\sslash G$  with  a polynomial map $\chi:\fg \to \CC^N$.
We associate to  a fixed   $v\in \fg$ a pair of (very simple) dgla's, $C^\bullet$ and $B^\bullet$
(\ref{first_examples}, \ref{toy_model}), whose Maurer--Cartan functors satisfy
$\MC_{C^\bullet}=\fg$, $\MC_{B^\bullet}=\CC^N$. 
Motivated by \cite{kontsevich_def_quant} \S 4.2,
 we construct an $L_\infty$-morphism
$h_\infty: C^\bullet\to B^\bullet$, such that $\MC(h_\infty)=\chi$ (after some identifications).

A suitable modification of  $h_\infty$ gives an 
$L_\infty$-description of the Hitchin map, described  in Section 4.1, where we prove {\bf Theorem B}.
That in turn gives information about obstructions to smoothness for the functor 
$\Def_{\scC^\bullet}$. These are considered in Section 4.2, together with the proof of {\bf Theorem A}. For details 
about obstruction calculus we refer to \cite{fantechi_manetti_1} and \cite{mansea}.

Our results are  the  natural generalisation of   \cite{martinengo}, \S 7,  where  the case of $G=GL(n,\CC)$ is treated by
  ingenuous use of   powers and traces of matrices.

      \subsection{Acknowledgements.}
 I am grateful to  Anya Kazanova for encouragement, and to  E.Cattani, C.Daenzer, S.Guerra, E.Markman, T.Pantev
and J.Stasheff
for support and helpful comments.
 I would like to thank ICTP and the organisers of the \emph{Conference on Hodge Theory and Related Topics} (2010), during which 
an early version of this material was completed.
This  work was partially supported by a CERES-CEI Research Fellowship (FP-7), Grant agreement 229660.

\section{Preliminaries}
  \subsection{Notation and Conventions}
The ground field is $\CC$.
We denote by $\Art$ the category of local Artin  $\CC$-algebras with residue field $\CC$, and denote by $\fm_A$  the maximal ideal
of $A\in \Art$. We denote by $\textrm{Fun}(\Art,\sets)$ the category of functors from $\Art$ to $\sets$, and use
``morphism of functors'' and ``natural transformations'' interchangeably.
We use the standard acronym ``dgla''  for a ``differential graded Lie algebra''.
 If $V^\bullet$ is a graded vector space, we denote by $V[n]$ its shift by $n$, i.e., $V[n]^i=V^{n+i}$.
We denote by  $T(V)$ the tensor algebra and by
  $S(V)=\bigoplus_{k\geq 0} S^k(V)$ the symmetric algebra of a (graded) vector space $V$.
%The natural projection is denoted by $\pi: T(V)\to S(V)$.
 The same notation is used for the
underlying vector spaces of the corresponding
 coalgebras, but   we use $S_c(V)$ and $T_c(V)$ when we want to emphasise the coalgebra structure.
  The \emph{reduced} symmetric, resp. tensor (co)algebra is denoted
$\overline{S(V)}=\bigoplus_{k\geq 1}S^k(V)$, resp. $\overline{T(V)}$. 
We denote by $\cdot$  the multiplication in $S(V)$.
%, and by
%$\Delta$ the  comultiplication in $S_c(V)$ and $T_c(V)$. 
By $S(k,n-k)$ we denote the $(k,n-k)$ unshuffles:
the  permutations $\sigma\in \Sigma_n$, satisfying  $\sigma(i)<\sigma(i+1)$ for all $i\neq k$.

Next, $G$ is a   complex reductive  Lie group of rank $N$  and $\fg = \textrm{Lie}(G)$. 
We  use fixed homogeneous generators, $\{p_i, i\in E\}$, of the ring of $G$-invariant polynomials on $\fg$.
The degrees of the invariant polynomials
are $d_i=\deg p_i$, so the exponents of $\fg$ are $d_i-1$.
The adjoint quotient map will always be given in terms of this basis, i.e., $\chi:\fg\to \CC^N\simeq \fg\sslash G$.
% We denote by $\scP_{dk}$ various polarisation maps, introduced in
%section \ref{polarisation}.

The base manifold $X$ is assumed to  be compact and complex.
For a holomorphic principal bundle, $\bP$, we denote by $\ad \bP$ its  associated bundle of Lie algebras, 
$\ad \bP=\bP\times_{\ad}\fg$. We denote by $\Omega^p_X$ the sheaf of holomorphic $p$-forms on $X$, and by
 $A^{p,q}$  the global sections of the sheaf $\scA^{p,q}$ of complex differential forms of type $(p,q)$.

We use $\cB$ for the Hitchin base, $\scB^\bullet$ for the abelian dgla governing
the deformations of an element of $\cB$, and  $B^\bullet:=\CC^N[-1]$  for the ``toy model'' of $\scB^\bullet$, see Section \ref{first_examples}.
We use $\scC^\bullet$ for the dgla controlling the deformations of a Higgs pair 
$(\bP,\theta)$ (see Section \ref{dgla}) and $C^\bullet$ for its ``toy version''
$\fg\otimes \CC[\eps]/\eps^2$.

  \subsection{Differential Graded Lie Algebras}\label{dgla}
Since there exist numerous introductory  references for this material
(\cite{goldman-millson}, \cite{gm_kur},  \cite{maurer}, \cite{manetti_complex}, \cite{mansea}),
we present here only  the basic definitions, without attempting to
motivate them in any way.
A \emph{differential graded Lie algebra} (dgla) is a triple $(\scC^\bullet,d,[\ , \ ])$. Here
$\scC^\bullet =\bigoplus_{i\in\NN} \scC^i[-i]$ is a graded vector space, endowed with a bracket
$[\ ,\ ]: \scC^i\times \scC^j\to \scC^{i+j}$. The bracket is graded skew-symmetric and
 satisfies a graded Jacobi identity. Finally, 
 $d:\scC\to \scC[1]$ is a differential ($d^2=0$), which is a degree 1 derivation of the bracket.
To a dgla $\scC^\bullet$ we associate a  Maurer--Cartan functor $\MC_{\scC^\bullet}: \Art\to \sets$, 
\[
\MC_{\scC^\bullet}(A)=\left\{ u\in \scC^1\otimes \fm_A \left|  du+\frac{1}{2}[u,u]=0 \right. \right\} 
\]
and a deformation functor $\Def_{\scC^\bullet}: \Art\to \sets$,
\[
\Def_{\scC^\bullet} (A)= \MC_{\scC^\bullet} (A)/ \exp(\scC^0\otimes \fm_A) . 
\]
The (gauge) action of $\exp(\scC^0\otimes \fm_A)$ on $\scC^1\otimes \fm_A$ is given by
\begin{equation} \label{gauge_action}
 \exp(\lambda): u\mapsto \exp(\ad \lambda)(u)+\frac{I-\exp(\ad\lambda)}{\ad\lambda}(d\lambda).
\end{equation}
Often  $\MC_{\scC^\bullet}(A)$ is considered as the set of objects of a groupoid (\emph{the Deligne groupoid}),  
which is the action groupoid for 
the gauge action on $\MC_{\scC^\bullet}(A)$
(\cite{goldman-millson}, \S 2.2). 
  \subsection{Examples}\label{first_examples}
Deformation problems are described by  deformation functors $\Def: \Art\to \sets$, and we say that a problem is
 governed (controlled) by a dgla $\scC^\bullet$, if there exists an 
isomorphism   $\Def_{\scC^\bullet}\simeq \Def$. A compendium of examples can be found in
 \cite{maurer}, \S 1 or in \cite{stasheff_schlessinger}. The controlling dgla is by no means unique, but
quasi-isomorphic dgla's have  isomorphic deformation functors (\cite{maurer}, Corollary 3.2). We give now a minimalistic (abelian) example, 
which will be used later.

Let $V$ be a finite-dimensional vector space, and $\xi\in V$. We consider the functor
$\Def_{\xi,V}:\Art\to \sets$ of embedded deformations of $\xi\in V$. That is, for any $A\in\Art$,
\[
 \Def_{\xi,V}(A) = \left\{ \sigma\in V\otimes A\left| \sigma = \xi\mod \fm_A\right. \right\}
= \left\{ \xi\right\} + V\otimes \fm_A \subset V\otimes A,
\]
with the obvious map on morphisms.
Then   $V[-1]$, a dgla with trivial bracket and trivial differentials, concentrated in degree $1$, controls the deformation problem.
Indeed, 
 $\MC_{V[-1]}(A)\equiv \Def_{V[-1]}(A)= V\otimes \fm_A$, which we write as $\MC_{V[-1]}= V = \Def_{V[-1]}$.
 The bijection
$\Def_{V[-1]}(A)\simeq \Def_{\xi,V}(A) $,
$s\mapsto \xi+s$, induces an isomorphism of functors $\Def_{V[-1]}\simeq \Def_{\xi,V}$.

Suppose  now $X$ is K\"ahler, and $V= H^0(X,F)$, 
for a holomorphic vector bundle   $F \to X$,  with $h^i(F)=0$ for $i\geq 1$.  
It is then easy to see that $\Def_{\xi,H^0(F)}$   is  isomorphic to the deformation functor of  the 
abelian dgla $\left(A^{0,\bullet}(F)[-1],\dbar_F\right)$, where $\dbar_F$ is the Dolbeault operator of $F$.
The isomorphism is  induced
by the canonical inclusion
 $H^0(X,F)\subset A^{0,0}(X,F)\subset A^{0,\bullet}(F)[-1]^1$.
The existence of such an inclusion relies on Hodge theory (see \cite{GH}, Chapter 0 \S 6, Chapter 1 \S 2),
and this is  where we use the K\"ahler condition.
To prove that the two dglas
$H^0(X,F)[-1]$ and  $\left(A^{0,\bullet}(F)[-1],\dbar_F\right)$ are quasi-isomorphic, one can use the Hodge decomposition and
follow the general setup from \cite{gm_kur}, \S 2 or \cite{kontsevich_def}.

We denote by $\Def_\xi$ the functor of deformations of a section $\xi\in H^0(X,F)$. It is isomorphic to the deformation functor
of $\left(A^{0,\bullet}(F)[-1],\dbar_F\right)$ on an arbitrary $X$ and without the vanishing condition. There is a natural morphism
$\Def_{\xi,H^0(F)}\to \Def_{\xi}$.

We reserve special notation for two instances of this example, namely
$B^\bullet :=\CC^N[-1]$ and $\scB^\bullet:= \left(\oplus _{i} A^{0,\bullet}(S^{d_i}\Omega^1_X)[-1], \dbar_X\right)$.
We also use $\cB= \bigoplus_{i} H^0(X,S^{d_i}\Omega^1_X)$ for the
 Hitchin base.
  \subsection{L-infinity algebras: Motivation}  
The notion of $L_\infty$-algebra (strongly homotopy Lie algebra, Sugawara algebra) generalises the notion of a dgla
by relaxing the Jacoby identity,
and allowing  it be satisfied only ``up to homotopy'' (``BRST-exact term''), determined by ``higher brackets''.
For  a detailed  motivation to this (somewhat technical) subject and  its applications 
to geometry and physics  we  refer  to 
\cite{stasheff_lada}, \cite{kontsevich_def_quant}, \cite{stasheff_schlessinger} and the references therein.
Here we make some non-rigorous remarks along the lines of  \cite{kontsevich_def_quant}, \S 4 and
give the precise definitions (following  \cite{manetti_complex}, Chapter VIII) in the next subsection.

Suppose that we want to study (algebraically) a formal neighbourhood
of $0\in V$, where 
$V$ is a (possibly infinite-dimensional) vector space.
One way to do this is to  consider the reduced
cofree cocommutative coassociative coalgebra, cogenerated by $V$, that is,  
$\overline{C(V)}= \bigoplus_{n\geq 1}\left(V^{\otimes n}\right)^{\Sigma_n}\subset \overline{T_c(V)}$.
Indeed, if $V$ is \emph{finite-dimensional}, then $\overline{C(V)}^\vee $ is the maximal ideal of the  algebra of formal power series.
Next, a morphism $\overline{C(V)}\to \overline{C(W)}$ is determined,
by the universal property of cofree coalgebras, 
 by
a linear map $h: \overline{C(V)}\to W$, with homogeneous components $h^{(n)}: \left(V^{\otimes n}\right)^{\Sigma_n}\to W$,
which are closely related to the Taylor coefficients of $h$. Indeed, the Taylor coefficients of $h$
are symmetric multilinear maps $h_n=\partial^n h: V^{\otimes n}\to W$. They  factor through the quotient $S^n(V)$, and are
carried to
 $h^{(n)}$  under the identification   $S^n(V)\simeq \left(V^{\otimes n}\right)^{\Sigma_n}$.
All of this can be done with graded vector spaces as well.

An $L_\infty$-structure  on a graded vector space $V^\bullet$ is the data of  a degree $+1$ coderivation $Q$ of  the coalgebra $\overline{C_c(V[1])}$,
satisfying $Q^2=0$, i.e., a codifferential. This is thought of as an odd vector field on the formal graded manifold $(V[1],0)$.
Its Taylor coefficients $q_n= \partial^n Q: \overline{S^n_c(V[1])}\to V[1]$ can be considered, by 
the d\'ecalage isomorphism $S^n(V[1])\simeq \Lambda^n(V)[n]$,
 as maps
$\mu_n\in \mhom^{2-n}(\Lambda^n V^\bullet,V^\bullet)$. 
  \subsection{L-infinity  algebras: Definitions}
An \emph{$L_\infty$-algebra structure} $(V^\bullet, q)$ on
 a   graded vector space $V^\bullet$ is   a collection of linear maps   
 $q_k\in \mhom^1(S_c^k (V[1]),V[1]),$ $k\geq 1$, such that the natural extension
of $q=\sum_k q_k$  to a degree $+1$ coderivation $Q$ on $\overline{S_c(V[1])}$
is  a codifferential,  i.e.,  $Q^2=0$. We recall (\cite{manetti_complex}, Corollary VIII.34) that
\[
 Q(v_1\cdot\ldots \cdot v_n)=\sum_{k=1}^n \sum_{\sigma\in S(k,n-k)}\epsilon(\sigma) q_k(v_{\sigma (1)}\cdot \ldots\cdot v_{\sigma (k)})\cdot v_{\sigma(k+1)}\cdot \ldots \cdot v_{\sigma (n)} .
\]
A dgla is  an $L_\infty$-algebra with   $q_1(a)=-da$, $q_2(a\cdot b) =(-1)^{\deg a}[a,b]$, and $q_k=0$ for $k\geq 3$.
To an $L_\infty$-algebra $(V^\bullet,q)$ one associates a 
  Maurer--Cartan functor
$\MC_{V^\bullet}: \Art\to \sets$ 
\[
\MC_V(A)=\left\{  u\in V^1\otimes \fm_A\left| \sum_{k\geq 1} \frac{q_k(u^k)}{k!}=0 \right.  \right\}
\]
and a    deformation functor $\Def_{V^\bullet}$, $\Def_{V}(A)=\MC_V(A)/\sim_\textrm{homotopy}$.
We refer to \cite{manetti_complex},  IX and \cite{maurer}, \S 5
 for  the definition of homotopy equivalence between two
Maurer--Cartan elements.
We do not give it here,  since we shall work only with dgla's considered as $L_\infty$-algebras,
and for these  
gauge equivalence coincides with homotopy equivalence, see \cite{maurer}, Theorem 5.5.

A \emph{morphism} $h_\infty: (V,q)\to (W,\hat{q}) $ between
 two $L_\infty$-algebras  is  a sequence of linear maps
 $h_k\in \mhom^0(S^k_c V[1], W[1])$, $k\geq 1$, for which the induced coalgebra
 morphism
$H: \overline{S_cV[1]}\to \overline{S_c W[1]} $
is a chain map,
 i.e., satisfies $H\circ Q=\widehat{Q}\circ H$. If we denote the components of $Q$ and $H$ by
  $Q_k^n: S^k_c(V)\to S^n_c(V)$ and $H_k^n$, respectively,  then
the morphism condition  reads
\[
 \sum_{a=1}^{\infty} h_a\circ Q_{k}^a = \sum_{a=1}^{\infty}\hat{q}_a\circ H_{k}^a,
\]
for all $k\in \NN$ .

We emphasise   that
the category of dgla's  is a subcategory  of the category of $L_\infty$-algebras, but
it is \emph{not}  full.
For a dgla the only possibly non-zero components of $Q$ are  $Q_k^k$, $k\geq 1$ and $Q_k^{k-1}$, $k>1$. For an 
 \emph{abelian} dgla  $q_2=0=Q_k^{k-1}$.
 We spell  out the condition for an $L_\infty$-morphism $h_\infty:(V,q)\to (W,\hat{q})$
from a dgla to an abelian dgla.
% First off, we have
%
%\begin{equation}\label{codif1}
%Q_k^{k}(s_1\cdot \ldots \cdot s_k) = \sum_{\sigma\in S(1,k-1)}\eps(\sigma) q_1(s_{\sigma_1})\cdot s_{\sigma_2}\cdot \ldots  \cdot s_{\sigma_k}
%\end{equation} 
%\begin{equation}\label{codif2}
%Q_k^{k-1}(s_1\cdot \ldots \cdot s_k)   =\sum_{\sigma\in S(2,k-2)}\eps(\sigma) q_2(s_{\sigma_1}\cdot s_{\sigma_2})\cdot s_{\sigma_3}\cdot \ldots  \cdot s_{\sigma_k}.
%\end{equation}
%
The $\{h_k\}$ determine an $L_\infty$-morphism if
\begin{equation}\label{mor1}
 h_1\circ q_1=\hat{q}_1\circ h_1,
\end{equation}
which says that $h_1$ is a morphism of complexes, and
\begin{equation}\label{mor2}
 h_k\circ Q_k^k + h_{k-1}\circ Q_k^{k-1} = \hat{q}_1\circ h_k,\  k\geq 2.
\end{equation}
The last condition, when evaluated on homogeneous elements $s_1,\ldots, s_k$ reads

\begin{equation}\label{morphism}
h_k\left(  - \sum_{\sigma\in S(1,k-1)}\eps(\sigma) d(s_{\sigma_1})\cdot s_{\sigma_2}\cdot \ldots  \cdot s_{\sigma_k} \right) +
\end{equation} 
$$h_{k-1}\left(  \sum_{\sigma\in S(2,k-2)}\eps(\sigma)(-1)^{\deg s_{\sigma_1}} [s_{\sigma_1}, s_{\sigma_2}]\cdot s_{\sigma_3}\cdot \ldots  \cdot s_{\sigma_k}\right) = $$
$$-dh_k(s_1\cdot\ldots \cdot s_k).  $$ 
It expresses the failure of $h_{k-1}$ to preserve the bracket in terms of a homotopy given by $h_k$.

Finally, $(V,q)\mapsto \MC_V$ determines a  functor $\MC: L_\infty\to Fun(\Art,\sets)$, whose action on morphisms is given by
 sending an $L_\infty$-morphism $h_\infty\in \mhom_{L_{\infty}}(V,W)$ to a natural transformation $\MC(h_\infty):\MC_V\to\MC_W$,
and,  for each $A\in \Art$, 
\begin{equation}\label{functoriality}
\MC(h_\infty)(A): MC_V(A)\ni x\longmapsto  \sum_{k=1}^{\infty}\frac{1}{k!}h_k(x^k)\in \MC_W(A).
\end{equation}
This descends to a natural transformation $\Def(h_\infty): \Def_V\to \Def_W$. For more details, see, e.g., \cite{manetti_complex}.

  \subsection{Deformation functors for Higgs bundles}
As already stated,  for us a Higgs bundle (Higgs pair) is a pair $(\bP,\theta)$, $\theta\in H^0(X,\ad\bP\otimes \Omega^1_X)$,
$\theta\wedge\theta=0$. We use the term \emph{$L$-valued Higgs bundle} if instead  $\theta\in H^0(X,\ad \bP\otimes L)$,
 for some vector bundle $L\to X$  (as in \cite{don-gaits}, \S 17).

Infinitesimal deformations of Higgs bundles have been  studied extensively.  Biswas and Ramanan (\cite{Biswas-Ramanan}) discussed
 the functor of deformations $\Def_{(\bP,\theta)}$ of a Higgs pair $(\bP,\theta)$   for $\dim X=1$,  and 
identified a  deformation complex, while in
   \cite{biswas_gl} a deformation complex is given for $G=GL(n,\CC)$ and a higher-dimensional (varying) base $X$.
  For arbitrary (fixed) compact K\"ahler $X$ and
arbitrary reductive $G$, the dgla  controlling the  deformations of  $(\bP,\theta)$  
 is   
\begin{equation}\label{higgs_dgla}
 \scC^\bullet =\bigoplus_{p+r=\bullet}A^{0,p}(X,\ad\bP\otimes \Omega_X^r),
\end{equation} 
with differential $\dbar_{\bP} + \ad\theta$, see \cite{simpson_hodge} \S 9, \cite{hbls} \S 2,  \cite{moduli2}, \S 10.  The complex $\scC^\bullet$ is the Dolbeault resolution of the  complex 
from \cite{Biswas-Ramanan}, \cite{biswas_gl}.
For the case of $G=GL(n,\CC)$ and $L$-valued pairs, one replaces $\Omega^q_X$ with $\Lambda^q(L)$, see \cite{martinengo}.
We note that the isomorphism
$\Def_{\scC^\bullet}(A)\simeq \Def_{(P,\theta)}(A)$ is obtained by mapping $[(s^{1,0},s^{0,1})]$ to $(\ker (\dbar+s^{0,1}),
\theta+s^{1,0})$, see \cite{biswas_gl}, \cite{moduli2}, \cite{martinengo}, .

We set $H(P,\theta):=\chi(\theta) \equiv \oplus_i p_i(\theta)\in \cB$. Using the notation from \S \ref{first_examples},
 define the (infinitesimal) Hitchin map as a morphism (natural transformation) of deformation functors 
\[
H: \Def_{(\bP,\theta)} \to \Def_{H(\bP,\theta)}
\]
by $H(A)(P_A,\theta_A)= \chi(\theta_A)$, $A\in\Art$, and the obvious map on morphisms. See also
\cite{Biswas-Ramanan}, Remark 2.8 (iv) or \cite{don-gaits}, \S 17.7.
While the two deformation functors at hand are controlled by dgla's
\begin{equation}\label{defo_functors}
\Def_{\scB^\bullet}\simeq \Def_{H(E,\theta)} \textrm{ and } \Def_{\scC^\bullet}\simeq \Def_{(\bP,\theta)},
\end{equation}
$H$ is not a dgla morphism, unless $G=(\CC^\times)^N$,
since it is not even linear. It is, however, induced by an $L_\infty$-morphism, as we intend to show.

In this note we are concerned with infinitesimal considerations only, but we remark that the
 coarse moduli spaces of semi-stable Higgs bundles (whenever they exist)
carry an amazingly rich geometry. We refer to \cite{hitchin_sd}, \cite{hitchin_sb}, \cite{hbls}, \cite{moduli2}, 
\cite{simpson_hodge} and \cite{don-gaits} for insight and  discussion of global questions.

  \section{The Adjoint Quotient in L-infinity terms}\label{adjoint_quotient}
%As stated in the introduction, it is generally expected that morphisms of moduli spaces should be induced by
%$L_\infty$-morphisms of the controlling dgla's.
% In Section \ref{dgla} we introduced dgla's  controlling the deformations of a Higgs bundle $(\bP,\theta)$
%and the deformations of a global section, $\xi$, of $\oplus_i S^{m_i}\Omega_X^1$.
%The main goal of this note is to describe an $L_\infty$-morphism, $h_\infty$, between these two dgla's,
%with the property that $\Def(h_\infty)$ coincides with the
%Hitchin map $H$ between the corresponding deformation functors. 
    \subsection{Toy Model}\label{toy_model}
If one sees the Hitchin map $H$    as a 
``global analogue'' of the adjoint quotient  $\chi:\fg\to \CC^N\simeq \fg\sslash G$, then the Higgs field $\theta$ should  be 
regarded
as a ``global analogue'' of an element $v\in \fg$.
In the present section we describe the morphism $\chi$
in  $L_\infty$ terms, and 
  in Section \ref{hitchin_map} we  modify suitably this  ``toy model'' to obtain an $L_\infty$-description of $H$.

Consider first the dgla  $C^\bullet := \fg\otimes \CC[\eps]/\eps^2 =\fg\oplus \fg[-1]$, with differential
$d_0 = \eps \ad v$. Since $d_1=0$ and $[C^1,C^1]=0$, we have $\MC_{C^\bullet}=\fg$, i.e.,  $\MC_{C^\bullet}(A) = \fg\otimes \fm_A$,
for all $A\in\Art$.
Moreover, the formula (\ref{gauge_action}) for the gauge action  reduces to 
$(\lambda, a)\mapsto e^{\ad \lambda}(v+a)-v$. We also recall from \S \ref{first_examples} the
dgla $B^\bullet=\CC^N[-1]$, with $\MC_{B^\bullet}=\CC^N$.
To see why is it appropriate to consider $C^\bullet$, we introduce
 the functor $\Def_{v,\fg,G}:\Art\to \sets$, 
\[
\Def_{v,\fg,G}(A)= \Def_{v,\fg}(A)/\exp(\fg\otimes\fm_A), 
\]
with the obvious transformation under morphisms of the coefficient ring.
That is, $\Def_{v,\fg,G}(A)$ is the quotient of the affine subspace $\left\{ v\right\}+\fg\otimes\fm_A \subset \fg\otimes A$
under the natural affine action of $\exp(\fg\otimes\fm_A)$ 
(\cite{goldman-millson}, \S 4.2), 
 which we briefly recall.
There is a natural Lie bracket on $\fg\otimes A$, obtained by extending the bracket on $\fg$.
The adjoint action of $G$  on $\fg$ extends  to
an action on $\exp(\fg\otimes\fm_A)$, and we denote by $G_A$ the semidirect product $\exp(\fg\otimes\fm_A)\rtimes G$.
More intrinsically, if we consider $G$ as the group of $\CC$-points of a $\CC$-algebraic group
${\bf G}$, then $G_A=\bG(A)$.
% can be identified with the $A$-points of ${\bf G}$.
%We refer the reader to \cite{goldman-millson}, Section 4.2  for more details on the structure of this Lie group.
The subgroup $\exp(\fg\otimes \fm_A)\subset G_A$ acts, via the adjoint representation, 
on $\fg\otimes A$, and preserves the affine subspace $\{v\}+ \fg\otimes \fm_A$. The  affine
action on $\fg\otimes \fm_A$ is $(\lambda, a)\mapsto e^{\ad \lambda}(v+a)-v$.
%

%Our next step is to exhibit two differential graded Lie algebras, whose deformation functors are
%isomorphic to the functors $\Def_{v,\fg,G}$ and $\Def_{\xi,\CC^N}$, respectively.
  
%
 Thus we have a bijection
$\Def_{C^\bullet}(A)\simeq \Def_{v,\fg,G}(A)$, 
$ a\mapsto v+a$
which induces an isomorphism 
$\Def_{C^\bullet}\simeq \Def_{v,\fg,G}$, as all constructions are natural in the coefficient ring.
Notice that $H^0(C^\bullet)$  is the centraliser of $v\in\fg$, so the functor $\Def_{C^\bullet}$
need not be representable. However, we have the following:
\begin{Proposition}\label{hull}
 Let $\cK\subset \fg$ be a linear complement to $\textrm{Im} (\ad_v)\subset \fg$,
and let $\widehat{\cO}_{(\cK,0)}$ be its completed local ring at the origin.
 Then the functor
$\mhom_{alg}(\widehat{\cO}_{K,0},\textrm{  } )$ is a hull for $\Def_{C^\bullet}$.
\end{Proposition}
\begin{proof}
 By Theorem 1.1. of \cite{gm_kur},  if a dgla $C^\bullet$ is equipped with a splitting
 and has finite-dimensional
 $H^k(C^\bullet)$, $k=0, 1$, 
 then it
 admits a hull $Kur\to \Def_{C^\bullet}$ by formal Kuranishi theory.
 In Theorem 2.3, \cite{gm_kur} it is 
shown that  under certain  topological conditions 
$Kur= Hom_{alg}(\widehat{\cO}_{(\cK,0)},\ )$, where  $(\cK,0)$ is the germ of a complex-analytic
space (Kuranishi space) and $\widehat{\cO}$ is  its completed local ring. 
In our case, $C^1=\fg$, $d_1=0$ and $[C^1,C^1]=0$, so
by Theorems 2.6 and  1.1, \cite{gm_kur} $\cK$ exists and can be taken to be any linear complement  to the coboundaries,
i.e., any  linear complement $\textrm{Im}\ad_v\subset \fg$.
\end{proof}

Our next  step is  to construct an $L_\infty$-morphism 
$h_\infty: C^\bullet\to B^\bullet=\CC^N[-1]$, such that
$\MC(h_\infty): \MC_{C^\bullet}=\fg\to \MC_{B^\bullet}=\CC^N$ gives  
the adjoint quotient.
This involves two ingredients.
First, as $\chi$ is given by homogeneous polynomials,  Taylor's formula can be expressed conveniently
by polarisation. Second, the derivatives of $G$-invariant polynomials satisfy extra
relations.
We discuss these technical properties in Section \ref{polarisation}, and
construct the promised $L_\infty$-morphism in
 Section \ref{adjoint_L_inf}.
  \subsection{Polarisation and Invariant Polynomials}\label{polarisation}
 Let $V$ be  a finite-dimensional vector space.
We have, for each $d,k\in \NN$,  a linear map
\[
\scP_{d}^{k,d-k}=\partial^k : S^d(V^\vee)\longrightarrow T^k(V^\vee)^{\Sigma_k} \otimes S^{d-k}(V^\vee),\ p\mapsto \partial^k p . 
 \]
That is, 
$$ \scP_d^{k,d-k} (p)(X_1\otimes \ldots\otimes X_k\otimes v_1\cdot \ldots \cdot v_{d-k}) = \cL_{X_1}\ldots\cL_{X_k}(p)(v_1\cdot \ldots \cdot v_{d-k}) , $$
where $\cL_X$ denotes  Lie derivative.
 %We shall occasionally denote by
%$\scP_{dk}(p)_v$ the corresponding map $T^k(V)\to \CC$.
%
%
%
%
Differently put,  $\scP_d^{k,d-k}(p)(X_1\otimes \ldots\otimes X_k\otimes v^{d-k})$ is   the coefficient in front of
$t_1\ldots t_k$ in the Taylor expansion of $p(v+\sum t_iX_i)$.
For example, if $V=\fg\fl(r,\CC)$ and  $p(A)=\textrm{tr} A^d$, then 
\[
\scP_d^{k,d-k}(p)(X_1\otimes \ldots\otimes X_k\otimes A^{d-k})=\frac{d!}{(d-k)!}\textrm{tr}(X_1\ldots  X_k A^{d-k}).\]
In particular, 
% $\scP_d^{k,d-k}=\partial^k=0$ for $k>d$ and that 
 $\scP_d^{d,0}=\partial^d: S^d(V^\vee)\simeq \left(V^{\vee\otimes d}\right)^{\Sigma_d}$ is the usual polarisation map, 
identifying $\Sigma_d$ 
invariants and coinvariants, and
$p(X)= \frac{1}{d!}\partial^d p(X^{\otimes d})$.
More generally,
%
%$$\scP_{dk}(p)(X_1,\ldots,X_k;v) = \frac{1}{(d-k)!}\scP_{dd}(p)(v,\ldots ,v,X_1,\ldots,X_k).   $$
%
\[
 (\partial^k p)(X_1\otimes \ldots\otimes X_k\otimes v^{d-k}) = \frac{1}{(d-k)!}(\partial^d p)(X_1\otimes \ldots\otimes X_k\otimes v^{\otimes d-k}) 
\]
and by Taylor's formula
\begin{equation}\label{taylor}
p(v+X) - p(v) = \sum_{k=1}^\infty \frac{1}{k!} (\partial^k p)(X^{\otimes k}\otimes  v^{d-k}). 
\end{equation}
 We prove two technical lemmas.
\begin{Lemma}\label{lemma}
 Let  $p\in \CC[\fg]^G$ be a homogeneous $G$-invariant polynomial of degree $d$.
Then   $(\partial p)(\ad_X(v)\otimes v^{d-1}) =0$, for all $v, X\in \fg$.
\end{Lemma}
\begin{proof}
The statement that $\frac{d}{dt}p(v+t\ad_X(v))\vert_{t=0} =0$   is just an infinitesimal form of the 
$G$-invariance of $p$.
Alternatively, one  can write the above expression as $\frac{1}{(d-1)!}$ times
$$(\partial^d p)(\ad_X(v) \otimes v^{\otimes d-1} )=\left. \frac{1}{d}\frac{d}{dt}\left(\partial^d p \left(\left(Ad(e^{tX})v\right)^{\otimes d} \right)\right)\right|_{t=0} =0. $$
\end{proof}

\begin{Lemma}\label{factor}
 Let $V=\bigoplus_{i=0}^{k-1}V_i$,  $F\in T^d(V^\vee)^\Sigma$, and  $L\in \prod_i GL(V_i)$. The decomposition of $V$
induces a decomposition of $S^d(V^\vee)$, indexed by ordered partitions of $d$ of length $k$. The projection of $F\circ (L\otimes 1^{\otimes d-1})$ 
onto the subspace corresponding to $(d-k+1,1,\ldots,1)$ maps $v\otimes X_1\otimes\ldots \otimes X_{k-1}$ to
\[
 \frac{d!}{(d-k)!}F(L(v)\otimes X_1\otimes \ldots\otimes X_{k-1}\otimes v^{\otimes d-k}) + 
\]
$$\sum_{\sigma\in S(1,k-2)}\frac{d!}{(d-k+1)!}F(L(X_{\sigma(1)})\otimes X_{\sigma(2)}\otimes \ldots\otimes  X_{\sigma(k)}\otimes v^{\otimes d-k}).$$
\end{Lemma}
\begin{proof}
The proof amounts to expanding $F(L(v+\sum_{i}X_i),v+\sum_{i}X_i,\ldots,v+\sum_{i}X_i)$ in powers of $X_i$, and counting the number of
terms, containing exactly one of each $X_i$.
\end{proof}

\begin{corollary}\label{funny}
 Let  $p\in\CC[\fg]^G$ be a homogeneous $G$-invariant polynomial of degree $d$.
Let $2\leq k\leq d$, and let  $v,Y,X_1,\ldots,X_{k-1}\in \fg$. Then
\[
(\partial^k p)([Y,v]\otimes X_1  \ldots  X_{k-1}\otimes v^{d-k}) +
\sum_{\sigma\in S(1,k-2)} (\partial^{k-1}p)([Y,X_{\sigma_1}]\otimes \ldots\otimes X_{\sigma_{k-1}}\otimes v^{ d-k+1} )=0.
\]
\end{corollary}
\begin{proof}
We apply Lemma \ref{factor} to $F= (\partial^d p)$ and $L=\ad Y$ and use  Lemma \ref{lemma} to argue that $F\circ (L\otimes 1)$  is zero.
\end{proof}
%
%This formula is the key to the $L_\infty$-description of the adjoint quotient and the Hitchin map, and it clearly has an
%operadic origin.
%
%
%
%As is customary, we are will  use the same notation for invariant polynomials on $\fg$ and invariant polynomials
% on sections of $\ad \bP\otimes \Omega^1_X$ (or forms with coefficients therein).
%, with values in $H^0(S^\bullet\Omega^1_X)$.
In the next section we will apply the various 
operators $\scP_d^{k,d-k}$ to sections of $\cA^{0,\bullet}(\ad\bP\otimes S^k\Omega^1_X)$ without changing the notation.
 For example, given sections $s_i$ expressed locally as  $s_i=\alpha_i\otimes X_i$ 
%$\in A^{0,\bullet}(\ad\bP\otimes \Omega^1_X)$ are
%decomposable,
 and $v\in H^0(X,\ad P\otimes \Omega^1_X)$, we write
 \[(\partial^k p)(s_1\otimes \ldots\otimes  s_k\otimes v^{d-k} ) =  \alpha_1\wedge\ldots\wedge\alpha_k (\partial^k p)(X_1\otimes \ldots\otimes X_k\otimes v^{d-k}). \]

    \subsection{The L-infinity Morphism}\label{adjoint_L_inf}
The main result of this section is the following
\begin{Proposition}\label{Lie2}
Let 
$\bp_0: \overline{S^\bullet\left( C^\bullet [1]\right)}\to \overline{S^\bullet\left( C^1\right)}$ denote  the   homomorphism
induced by the projection  $\textrm{pr}_2: C^\bullet[1]=\fg[1]\oplus \fg\to C^1=\fg$.
The collection of maps	
\[
\xymatrix@1{\bigoplus_i h^{d_i}_k= (\partial^k p_i)(_{-}\otimes v^{d_i-k})\circ p_{0} :& S^k\left( C^\bullet [1]\right)\ar[r]& \  \CC^N  }
 \]
\[
 (a_1,b_1)\cdot \ldots \cdot (a_k,b_k)\longmapsto  \bigoplus_i (\partial^k p_i)(b_1\otimes \ldots\otimes  b_k\otimes v^{d_i-k})  
\]
  induces an $L_\infty$-morphism
$h_\infty: C^\bullet \to B^\bullet=\CC^N[-1]$. 
Under the identifications $\MC_{B^\bullet}\simeq \Def_{\chi(v),\CC^n}$ and $\MC_{C^\bullet}\simeq \Def_{v,\fg}$,
$\MC(h_\infty):\MC_{C^\bullet}\to \MC_{B^\bullet}$ coincides with $\chi:\fg\to \CC^N$.

\end{Proposition}
\begin{proof}
To show that this collection of maps determines an
 $L_\infty$-morphism, it suffices to verify that  for each fixed $d_i$, the maps $\{h_k^{d_i}\}$
determine an $L_\infty$-morphism  $C^\bullet\to \CC[-1]$. We prove this in Lemma \ref{Lie1}. Assuming that, 
let 
 $s=(0,b) \in \MC_{C^\bullet}(A)$, $b\in \fg\otimes \fm_A$ for $A\in \Art$. Then, by 
%formula 
(\ref{functoriality}),
 $\MC(h_\infty)(s)= \sum_{d=1}^{\infty}\frac{1}{d!}h_{\infty}(s^d)$,  
%  formula 
 which equals  $\oplus_i \left( p_i(v + b) -p_i(v)\right)= \chi(v+b)-\chi(v)$ by (\ref{taylor}).
The specified identifications amount to affine transformations translating the origin,
which carry $\MC(h_\infty)$ to the map
 $v+b\mapsto \chi(v+b)$, hence the last statement.
\end{proof}
\begin{Lemma}\label{Lie1}
 Let $p\in \CC[\fg]^G$ be a homogeneous polynomial of degree $d$.
The collection of maps	
\[
\xymatrix@1{h^{d}_k= (\partial^k p)(_{-}\otimes v^{d-k})\circ p_{0} :& S^k\left( C^\bullet [1]\right)\ar[r]& \  \CC  }
 \]
\[
 (a_1,b_1)\cdot \ldots \cdot (a_k,b_k)\longmapsto   (\partial^k p)(b_1\otimes \ldots\otimes  b_k\otimes v^{d-k})  
\]
  induces an $L_\infty$-morphism
$$h^{d}_\infty: C^\bullet\longrightarrow \CC[-1].$$
\end{Lemma}
\begin{proof}
We start with  condition  (\ref{mor1}). 
The  differentials of the two dgla's are, respectively, $\ad v$ and $0$, so 
 we have to show that, for any $s=(a,b)\in \fg^{\oplus 2}$,
$h^{d}_1([v,s]) =0$. But this means  $(\partial p)([v,b]\otimes v^{d-1} ) =0$, which is the conclusion of
 Lemma \ref{lemma}. We turn to (\ref{mor2}), whose right hand side is identically zero (since $B^\bullet$ is formal).
The left side is zero  on $S^k(C^1)$, since $[C^1,C^1]=0$. 
It is also zero on $S^r(C^0)\cdot S^{k-r}(C^1)$ for $r\geq 2$, 
since  $h^{d}_k$ factors through $p_{0}$.
So we only have to verify  (\ref{mor2})  on $C^0\cdot S^{k-1}(C^1)$,  in which  case  
$Q_k^{k-1}$ contributes via the bracket and $Q_k^k$ via $\ad v$.
Take    homogeneous elements  $s_j=(0,b_j),j\geq 2$ and $s_1 =(a,0)$.
In the first summand of (\ref{morphism}), 
unshuffles with   $\sigma(1)\neq 1$  give zero, while  $\sigma(1)=1$ means  $\sigma = id$, so we have
$h^{d}_k\left([v,s_1]\cdot s_2\cdot \ldots \cdot s_k \right) = (-1)(\partial^k p_i)([a,v]\otimes b_2\otimes \ldots\otimes b_k\otimes v^{d-k})$.
The second summand of (\ref{morphism}) is $h^{d}_{k-1}\circ Q_k^{k-1}(s_1\cdot\ldots\cdot s_k)$ and  the non-vanishing terms correspond 
to $(2,k-2)$ unshuffles $\sigma$, for which $\sigma(1)=1$. Hence  the summation is in fact over $(1,k-2)$ unshuffles and we have
\[ 
h^{d}_{k-1}\left(\sum_{\sigma\in S(1,k-2)}(-1)\epsilon(\sigma)[s_1,s_{\sigma(1)}]\cdot \ldots \cdot s_{\sigma(k-1)} \right) =\]
\[
(-1) \sum_{\sigma\in S(1,k-2)}(\partial^{k-1}p_i)\left([a,b_{\sigma(1)}]\otimes\ldots\otimes b_{\sigma(k-1)}\otimes v^{d-k+1} \right).
\]
Note  that  $C^1 = C^\bullet[1]^{0}$, so  $\epsilon(\sigma)=1$. 
The two summands add up to zero by  Corollary \ref{funny}.
\end{proof}
\section{The Hitchin Map  }\label{hitchin_map}
We prove now the two main results of  this note by suitably adapting the calculation of the previous section,
thus extending the  results of \cite{martinengo} to arbitrary reductive structure groups.
    \subsection{Proof of Theorem B}
\begin{proof}
To prove that the collection$\{h_k\}$  determines
an $L_\infty$-morphism, it suffices to prove that for each fixed homogeneous polynomial $p_i$ of degree $d_i$, the given collection of
maps induces an $L_\infty$-morphism $h_\infty^{d_i}:\scC^\bullet\to  A^{0,\bullet}(S^{d_i}\Omega^1_X)[-1] $. This is shown in Lemma \ref{hitchin1} below.
Assuming that, suppose $s=(s',s'')\in \MC_{\scC^\bullet}(A)$, $A\in \Art$. By (\ref{functoriality}) $\Def(h_\infty)(s)= \sum_{d=1}^{\infty}\frac{1}{d!}h_{\infty}(s^d)$, which by
 formula (\ref{taylor}) equals $\oplus_i p_i(\theta + s') -p_i(\theta) =H(P_A,\theta_A)-H(P,\theta)$. 
This is exactly what we want to prove, in view of the identification (\ref{defo_functors}), which amounts to ``shifting the origin''.
\end{proof}
\begin{Lemma}\label{hitchin1}
Let
$p\in\CC[\fg]^G$  be a homogeneous polynomial  of  degree $d$.
 Let
 $\bp_0: \overline{S^\bullet(\scC^\bullet)}\to \overline{S^\bullet\left(A^{0,\bullet}(\ad\bP\otimes \Omega^1_X)\right)}$
denote the homomorphism
induced by $\bigoplus_{p+q=\bullet}s^{p,q}\mapsto s^{1,q}$, where $s^{p,q}\in  A^{0,q}(\ad\bP\otimes \Omega^p_X)$.
Then the collection of maps
\[
\xymatrix@1{h^d_k=(\partial^k p)(_{-}\otimes \theta^{d-k})\circ p_0:&S^k(\scC^\bullet[1])\ar[r]& A^{0,\bullet}(S^{d}\Omega^1_X)      }
\]
\[\bigoplus_{p_1,q_1} s_1^{p_1,q_1}\cdot \bigoplus_{p_2,q_2} s_2^{p_2,q_2}\cdot\ldots \cdot\bigoplus_{p_k,q_k} s_k^{p_k,q_k} 
\longmapsto \sum_{q_1,\ldots,q_k} (\partial^k p)( s_1^{1,q_1}\otimes \ldots\otimes  s_k^{1,q_k}\otimes \theta^{d-k}) 
\]
  induces an $L_\infty$-morphism
$$h^{d}_\infty: \scC^\bullet = \bigoplus _{r+s=\bullet} A^{0,r}(\ad \bP\otimes \Omega_X^s)\to   A^{0,\bullet}(S^{d}\Omega^1_X)[-1] .$$
\end{Lemma}
\begin{proof}
We check the conditions  (\ref{mor1}),(\ref{mor2}). 
The differentials are   $\dbar_\bP + \ad\theta$ and $\dbar_\bP$, so (\ref{mor1})
  is equivalent to
 $(\partial p)([\theta,s]\otimes\theta^{d-1}) =0 $, which holds by   Lemma \ref{lemma}.
  Next assume $k\geq 2$. 
Since by definition $h_k^d$ factors through $\bp_0$,
 both sides of (\ref{mor2}) are identically zero, except possibly  for \emph{two} cases.
Case 1: when evaluated on $S^k\left( A^{0,\bullet}(\ad P\otimes \Omega^1)\right)$ and Case 2: when 
evaluated on $A^{0,\bullet}(\ad \bP)\cdot S^{k-1}\left( A^{0,\bullet}(\ad \bP\otimes \Omega^1)\right) $.
Notice that  in  $\scC^\bullet[1]$, 	  the degree of a homogeneous element in $A^{0,n}(\ad P\otimes \Omega^1_X)$ is $n$.
We start with Case 1, evaluating on decomposable homogeneous elements $s_i=\alpha_i\otimes X_i$, $i=1\ldots k$. 
Since $[s_{\sigma 1}, s_{\sigma 2}]$ and $\ad\theta(s_{\sigma 1})$ belong to $A^{0,\bullet}(\ad \bP\otimes \Omega^2_X)$, they do not
contribute to the left side of (\ref{morphism}).
And since $\sum_{\sigma\in S(1,k-1)}\epsilon(\sigma)\dbar(\alpha_{\sigma (1)})\wedge\ldots \wedge\alpha_{\sigma (k)}=$ $\dbar (\alpha_1\wedge\ldots\wedge\alpha_k)$,
the left side of (\ref{morphism}) gives
\[
 - \dbar (\alpha_1\wedge\ldots\wedge\alpha_k)\otimes (\partial^k p)(X_1\otimes \ldots\otimes X_k\otimes \theta^{d-k}) = \hat{q}_1\circ h_k^{d}(s_1\cdot\ldots\cdot  s_k),
\]
which we wanted to show.
Next we proceed to Case 2, and  take decomposable homogeneous elements $s_i=\alpha_i\otimes X_i$, $s_1\in A^{0,\bullet}(\ad \bP)$, 
$s_2,\ldots, s_k\in A^{0,\bullet}(\ad \bP\otimes \Omega^1_X)$. The right hand side of (\ref{morphism}) is
zero on their product, so we just compute the left side.
The terms  with  $\sigma(1)\neq 1$  are identically zero, and  $\sigma(1)=1$ implies  $\sigma = id$, so we obtain
\[
h^{d}_k\left([\theta,s_1]\cdot s_2\cdot \ldots \cdot s_k \right) = (-1)^{\deg s_1}\alpha_1\wedge\ldots\wedge \alpha_k (\partial^k p)([X_1,\theta]\otimes X_2\otimes\ldots\otimes X_k\otimes\theta^{d-k}).
\]
 The non-vanishing contributions
from $h^{d}_{k-1}\circ Q_k^{k-1}$ in (\ref{morphism})
correspond 
to $(2,k-2)$ unshuffles for which $\sigma_1=1$, so the summation is in fact over $(1,k-2)$ unshuffles and we have
$$ 
h^{d}_{k-1}\left(\sum_{\sigma\in S(1,k-2)}(-1)^{\deg s_1}\epsilon(\sigma)[s_1,s_{\sigma(1)}]\cdot \ldots \cdot s_{\sigma(k-1)} \right).
$$
By the shift, 
  the 
Koszul sign  is  traded for reordering the forms and we get
\[
(-1)^{\deg s_1}\alpha_1\wedge\ldots \wedge \alpha_k\sum_{\sigma\in S(1,k-1)}(\partial^{k-1}p)([X_1,X_{\sigma(1)}]\otimes \ldots \otimes X_{\sigma(k-1)}\otimes \theta^{d-k+1}).
\]
Then the sum of the two terms  is zero by Corollary \ref{funny}. 
\end{proof}
      \subsection{Obstructions to smoothness}\label{obstructions}
While   Higgs bundles on  \emph{curves} have been extensively studied, fairly little is known about their moduli if
$\dim X>1$, apart from the general results of \cite{moduli2}, partially due to scarcity of examples. By  formality  (\cite{hbls}, Lemma 2.2),
Simpson's moduli spaces have at most quadratic singularities (\cite{moduli2}, Theorem 10.4). It is known that whenever $H^2(\scC^\bullet)=0$, the functor $\Def_{\scC^\bullet}$
is smooth (the representing complete local algebra is regular), see \cite{Biswas-Ramanan}, Theorem 3.1, 
\cite{biswas_schumacher_geom_higgs_moduli} Proposition 3.7, \cite{biswas_gl} Remark 2.8. We recall now the description of
the obstruction space $O_{\Def_{\scC^\bullet}}\subset H^2(\scC^\bullet)$.

Recall (\cite{fantechi_manetti_1}, \cite{mansea}, \S 4) that an \emph{obstruction theory} for a deformation
functor $F:\Art\to\sets$ is a pair $(V,v)$. Here  $V$ is a vector space (\emph{obstruction space}), and $v$ assigns to  any small extension
$\xymatrix@1{e: 0\ar[r]&M\ar[r]& B\ar[r]& A\ar[r]&0} $, an \emph{obstruction map} $v_e: F(A)\to V\otimes M$, respecting base change,
with $\textrm{Im}\left( F(B)\to F(A)\right)\subset \ker v_e$. The obstruction theory is \emph{complete},  if this containment is
an equality.
 A \emph{universal} obstruction theory is an obstruction theory $(O_F,o)$, admitting a unique morphism to any other
obstruction theory $(V,v)$. The vector space $O_F$ is called \emph{the obstruction space} of $F$.
\begin{proof}[{\bf Proof of Theorem A}]
The proof is  essentially a standard argument in deformation theory, and can be considered as a form of the
so-called ``Kodaira principle''.
By \cite{mansea}, Theorem 4.6 and Corollary 4.8 (see also \cite{fantechi_manetti_1}), 	any deformation functor $F$
admits a universal obstruction theory, and, if $(V,v)$ is any complete obstruction theory, then  $O_F$ is isomorphic to the space, generated by 
$v_e(F(A))$, where $e$ ranges over
all
principal (i.e., with $M=\CC$) small extensions. By Example 4.4, \cite{mansea}, for any dgla $L$,  the functor $F=\MC_L$ admits a complete obstruction theory
$(H^2(L), v)$. Here the map $v_e: \MC_L(A)\to H^2(L)\otimes M$ is defined by $v_e(x)=[h]$, where $h=d\widetilde{x}+\frac{1}{2}[\widetilde{x},\widetilde{x}]$, and 
$\widetilde{x}\in \MC_L(B)$ is a lift of $x\in \MC_L(A)$. Also, by Corollary 4.13, \cite{mansea}, the functors $\MC_L$ and $\Def_L$ have isomorphic 
obstruction theories. In particular, $O_{\Def_L}\subset H^2(L)$ and  if $L$  \emph{abelian}, then $O_{\Def_L}=(0)$.
Now consider  the abelian dgla $\scB^\bullet$ and
 $h_\infty:\scC^\bullet\to \scB^\bullet$. 
By equation (\ref{mor1}),  $h_1$ is a morphism of complexes, and one can show 
(\cite{maurer}) that $H^2(h_1)$ is a morphism of obstruction spaces, hence the result.

There is a more direct argument if $X$ is K\"ahler and
  $G$ is semi-simple (so that $d_i=0$ is not an exponent), or 
if $H^1(X,\cO_X)=0$. Indeed, in that case
  $\scB^\bullet\simeq_{qis} \CC^N[-1]$ (see Section \ref{first_examples}), so $H^2(\scB^\bullet)=(0)= H^2(h_1)(O_{\scC^\bullet})$.
\end{proof}

\bibliographystyle{alpha}
\bibliography{biblio}
\end{document}